\documentclass[12pt,reqno]{amsart}
\usepackage{amssymb, amsfonts, amsbsy, latexsym, epsfig, color}
\usepackage{enumerate}

\newtheorem{Thm}{Theorem}[section]

\newtheorem{corollary}[Thm]{Corollary}

\newtheorem{proposition}[Thm]{Proposition}
\newtheorem{definition}[Thm]{Definition}
\newtheorem{remark}[Thm]{Remark}

\newtheorem{example}[Thm]{Example}

\newtheorem{theorem}[Thm]{Theorem}

\def\NN{\mathbb{N}}

\def\RR{\mathbb{R}}

\def\cH{{\mathcal{H}}}

\newcommand{\spann}{\mbox{\rm span}}
\newcommand{\norm}[1]{\|#1\|}
\newcommand{\inner}[2]{\langle #1,#2 \rangle}

\begin{document}

\title{Weighted Fusion Frame Construction via Spectral Tetris}
\author[P.G. Casazza and J. Peterson
 ]{Peter G. Casazza and Jesse Peterson}
\address{Department of Mathematics, University
of Missouri, Columbia, MO 65211-4100}

\thanks{The authors were supported by
AFOSR F1ATA00183G003; NSF 1008183; and  NSF ATD 1042701}

\email{Casazzap@missouri.edu; 
 jdpq6c@mail.missouri.edu
}

\begin{abstract}
Fusion frames consist of a sequence of subspaces from a Hilbert space and corresponding positive weights so that the sum of weighted orthogonal projections onto these subspaces is an invertible operator on the space.  Given a spectrum for a desired fusion frame operator and dimensions for subspaces, one existing method for creating unit-weight fusion frames with these properties is the flexible and elementary procedure known as spectral tetris.  Despite the extensive literature on fusion frames, until now there has been no construction of fusion frames with prescribed weights.  In this paper we use spectral tetris to construct more general, arbitrarily weighted fusion frames.  Moreover, we provide necessary and sufficient conditions for when a desired fusion frame can be constructed via spectral tetris.
\end{abstract}

\maketitle

\section{introduction}
Fusion fames (initially frames of subspaces \cite{frames of subspaces}) consist of a sequence of subspaces from a Hilbert space and corresponding positive weights so that the sum of weighted orthogonal projections onto these subspaces is an invertible operator on the space.  Such structure provides a natural mathematical framework for hierachical data processing and lends itself to the design of systems which are robust against noise, data loss, and erasures \cite{B, CKL, KPCL}.  With such desirable properties, fusion frames have found application to problems in sensor networks and distributive processing just to name a few \cite{CF, CKL, G}.  For the interested reader, an in-depth listing of papers on fusion frames may be found at www.fusionframe.org (see also 
www.framerc.org).

Much work on fusion frames has involved constructions of frames with specialized properties \cite{CFMWZ09, MRS}, and specifically \cite{CFHWZ} describes efficient methods for constructing fusion frames with prescribed subspace dimensions and prescribed fusion frame operator eigenvalues via the so-called spectral tetris algorithm.  Each of these construction methods considers only the the unit-weight case.  Also, \cite{CFHWZ} provides necessary and sufficient conditions for when spectral tetris can generate such unit-weight fusion frames with the additional restriction that all fusion frame operator eigevnalues lie within $[2,\infty)$.   But until now, there has been no algorithm for constructing
fusion frames with prescribed weights.  
Our contribution in this paper is a construction of fusion frames with prescribed weights, prescribed subspace dimensions, and prescribed fusion frame spectra.  Further, we present necessary and sufficient conditions for when any fusion frame may be constructed via spectral tetris.  That is, we use spectral tetris to develop constructions for the most general classes of  fusion frames and give necessary and sufficient conditions for when this is possible.

We begin by reviewing frames, fusion frames, and the spectral tetris algorithm in Section \ref{background}.  Section \ref{construction} then investigates fusion frames constructed via spectral tetris and culminates in Theorem \ref{necessary and sufficient} which provides necessary and sufficient conditions for the existence of such fusion frames.  Finally, in Section \ref{concrete}, we modify the spectral tetris algorithm and provide the first concrete constructions of a fusion frame with a prescribed frame operator spectrum, prescribed subspace dimensions, and prescribed weights for each subspace.

\section{Background and Notation}\label{background}

\subsection{Spectral Tetris Frames}
A family of vectors $F=\{f_n\}_{n=1}^N$ is called a {\it frame} for an $M$-dimensional Hilbert space $\cH_M$ if there are constants $0<A\le B < \infty$ satisfying
\[ 
	A \|f\|^2 \le \sum_{n=1}^N |\langle f,f_n\rangle|^2 \le B\|f\|^2, \mbox{ for all } f\in \cH_M,
\]
where $A,B$ are called lower and upper {\it frame bounds} respectively.  Further, if $A=B$, we call such a frame a {\it tight frame}.  The focus of this paper will be real frames.  The {\em synthesis operator} of a real frame $\{f_n\}_{n=1}^N$ in $\RR^M$ is given by $F:\RR^N \rightarrow \RR^M$, $Ff=\sum_{n=1}^Nf(n)f_n$.  From a matrix perspective, $F$ is the $M \times N$ matrix whose columns are the $f_n$'s.  We make no distinction between the family of vectors $F=\{f_n\}_{n=1}^N$ and the induced $M\times N$ synthesis matrix.  Together with the {\it analysis operator} $F^*$, we have the {\it frame operator} $S=FF^*$.  Note the frame bounds $A,B$ are the largest and smallest eigenvalues of $S$.

The original {\it spectral tetris} algorithm developed in \cite{CFMWZ09} is a powerful tool for constructing sparse unit-norm tight frames.  This original algorithm has since been generalized and modified in many ways.  In \cite{CCHKP10}, a straightforward extension allows the construction of unit norm frames having a frame operator with desired eigenvalues $\{\lambda_m\}_{m=1}^M\subseteq [2,\infty)$.  The authors in \cite{CFHWZ} adapt the algorithm in the complex case to use discrete fourier transform matrices as blocks in the construction in order to extend the range of possible eigenvalues.  Most recently \cite{CHKWA} modified the algorithm to produce sparse frames with prescribed frame operator spectra and prescribed frame vector norms, naming this new process prescribed norm sprectral tetris construction (PNSTC).  As PNSTC is a generalized version of the original spectral tetris construction (STC) and includes STC a special case, we will refer to PNSTC simply as STC.  We recommend \cite{CCHKP10, CFMWZ09} for instructive examples concerning how the original algorithm constructs a desired synthesis matrix $F$ from singletons and $2\times 2$ blocks of the form
\[
	A(x)=	\begin{bmatrix}
				\sqrt{\frac{{x}}{2}}		&	\sqrt{\frac{{x}}{2}}	\\
				\sqrt{\frac{{1-x}}{2}}	&	\sqrt{\frac{{1-x}}{2}}
				\end{bmatrix}.
\]
Note the rows of $A(x)$ are orthogonal, the first row square sums to $x$, and columns have unit norm. 

We then recommend \cite{CHKWA} for useful examples on how STC uses singletons and $2\times 2$ blocks of the form
\begin{equation}\label{block}
	A(x,a_\ell,a_{\ell+1})=	\begin{bmatrix}
								\sqrt{\frac{x(a_\ell^2-y)}{x-y}}		&		\sqrt{\frac{x(x-a_\ell^2)}{x-y}}	\\
								\sqrt{\frac{y(x-a_\ell^2)}{x-y}}		&		\sqrt{\frac{y(a_\ell^2-y)}{x-y}}
								\end{bmatrix}
\end{equation}
where $y=a_\ell^2+a_{\ell+1}^2-x$ to generate a desired synthesis matrix with frame vector norms $\{a_n\}_{n=1}^N$.  Note the rows of $A(x,a_\ell,a_{\ell+1})$ are orthogonal, the first row square sums to $x$, and now the columns have norms $a_\ell$ and $a_{\ell+1}$.  Also in \cite{CHKWA}, the authors show such blocks $A(x,a_\ell,a_{\ell+1})$ exist if and only if
\begin{align}
	a_\ell^2+a_{\ell+1}^2\geq x >0 \mbox{ and } \label{bexist1}\\ 
	a_\ell^2,a_{\ell+1}^2>x \mbox{ or } a_\ell^2,a_{\ell+1}^2<x \label{bexist2},
\end{align}
and proceed to develop necessary and sufficient conditions for a frame to be constructed from spectral tetris.

These necessary and sufficient conditions are significant in that spectral tetris cannot construct all possible frames.  Given a sequence of vector norms and eigenvalues for a frame operator, the existence of such a frame is characterized by the Schur-Horn Theorem.  That is, these frames exist when the spectrum majorizes the norms squared.  While other algorithms \cite{eigensteps} are able to construct all such frames, spectral tetris's inability to do so is offset by its simplicity and the sparcity of its constructed frames.  Indeed it is this sparcity which allows STC to be useful in fusion frame construction.

Any frame which may be obtained via STC we call a {\em spectral tetris frame}.  Specifically, two sequences $\{a_n\}_{n=1}^N \subseteq (0,\infty)$ and $\{\lambda_m\}_{m=1}^M\subseteq (0,\infty)$ are {\it spectral tetris ready} if $\sum_{n=1}^N a_n^2= \sum_{m=1}^M \lambda_m$ and if there is a partition $1\leq n_1<\cdots<n_M=N$ such that for all $k=1,\ldots,M-1$
\begin{enumerate}[(a)]
	\item $\sum_{n=1}^{n_k}a_n^2 \leq \sum_{m=1}^k\lambda_m<\sum_{n=1}^{n_k+1}a_n^2$  and
	\item if  $\sum_{n=1}^{n_k}a_n^2< \sum_{m=1}^k\lambda_m$, then $n_{k+1}-n_k\geq 2$ and 
		\[
			a_{n_k+2}^2\geq \sum_{m=1}^k\lambda_m-\sum_{n=1}^{n_k}a_n^2.
		\]
\end{enumerate}

\begin{theorem} (Theorem 3.6 in \cite{CHKWA}) \label{ready}
STC can construct a frame with spectrum $\{\lambda_m\}_{m=1}^M\subseteq (0,\infty)$ and vector norms $\{a_n\}_{n=1}^N \subseteq (0,\infty)$ if and only if there exists a permutation of these sequences such that they are spectral tetris ready.
\end{theorem}

For convenience, the most general STC algorithm is given in table \ref{fig:STC}.  We note as STC creates an $M\times N$ synthesis matrix $F$, each step corresponds to a pair $(m,n)$.  We will refer to $(m,n)$ as the cursor location as this represents the row and column position in $F$ for which STC is creating entries.  Note that given spectral tetris ready sequences $\{a_n\}_{n=1}^N$ and $\{\lambda_m\}_{m=1}^M$, if the cursor is in row $r$, STC will insert a scalar ({\it singleton}) at the cursor location until column $\ell$ where
\begin{align}
	&\sum_{n=1}^{\ell-1} a_n^2 \leq \sum_{m=1}^r \lambda_m< \sum_{n=1}^{\ell} a_n^2. \label{endrow}
\end{align}
If the left side of (\ref{endrow}) holds as an equality, STC inserts another singleton and the cursor proceeds to row $r+1$, while if it does not hold as an equality, STC will create a $2\times 2$ block, $A=A(x,a_{\ell+1},a_{\ell+2})$ of the form (\ref{block}).  We refer to a column of a $2\times 2$ block as a {\it doubleton}.  

\begin{table}[h]
\centering
\framebox{
\begin{minipage}[h]{6.0in}
\vspace*{0.3cm}
{\sc \underline{STC: Spectral Tetris Construction}}

\vspace*{0.4cm}

{\bf Parameters:}\\[-3ex]
\begin{itemize}
\item Dimension $M\in\NN$.
\item Number of frame elements $N\in\NN$.
\item Eigenvalues $\{\lambda_m\}_{m=1}^M$  and norms of the frame vectors $\{a_n\}_{n=1}^N$ such that 
$(\lambda_m)_{m=1}^M$  and $(a_n^2)_{n=1}^N$ are spectral tetris ready.
\end{itemize}

{\bf Algorithm:}\\[-3ex]
\begin{itemize}
\item[1)] Set $n=1$.
\item[2)] For $m=1,\ldots,M$ do
\item[3)] \hspace*{0.5cm}Repeat
\item[4)] \hspace*{1cm}If $\lambda_m\geq a_n^2$ then
\item[5)] \hspace*{1.5cm}$f_n=a_ne_m$.
\item[6)] \hspace*{1.5cm}$\lambda_m=\lambda_m-a_n^2$.
\item[7)] \hspace*{1.5cm}$n=n+1$.
\item[8)] \hspace*{1cm}else
\item[9)] \hspace*{1.5cm}If $2\lambda_m=a_n^2+a_{n+1}^2$, then
\item[10)] \hspace*{2cm}$f_{n}=
\sqrt{\frac{\lambda_m}{2}}\cdot (e_m+e_{m+1}).$
\item[11)] \hspace*{2cm}$f_{n+1}=
\sqrt{\frac{\lambda_m}{2}}\cdot (e_m-e_{m+1}).$
\item[12)] \hspace*{1.5cm}else
\item[13)] \hspace*{2cm}$y=a_n^2+a_{n+1}^2-\lambda_m$.
\item[14)] \hspace*{2cm}$f_{n}=
\sqrt{\frac{\lambda_m(a_n^2-y)}{\lambda_m-y}}\cdot e_m
+\sqrt{\frac{y(\lambda_m-a_n^2)}{\lambda_m-y}}\cdot e_{m+1}.$
\item[15)] \hspace*{2cm}$f_{n+1}=
\sqrt{\frac{\lambda_n(\lambda_m-a_n^2)}{\lambda_m-y}}\cdot e_m
-\sqrt{\frac{y(a_n^2-y)}{\lambda_m-y}}\cdot e_{m+1}.$
\item[16)] \hspace*{1.5cm}end.
\item[17)] \hspace*{1.5cm}$\lambda_{m+1}=\lambda_{m+1}-(a_n^2+a_{n+1}^2-\lambda_m)$.
\item[18)] \hspace*{1.5cm}$\lambda_m=0$.
\item[19)] \hspace*{1.5cm}$n=n+2$.
\item[20)] \hspace*{1cm}end.
\item[21)] \hspace*{0.5cm}until $\lambda_m=0$.
\item[22)] end.
\end{itemize}

{\bf Output:}\\[-3ex]
\begin{itemize}
\item Frame $(f_n)_{n=1}^N\subseteq\RR^M$.
\end{itemize}
\vspace*{0.1cm}
\end{minipage}
}
\vspace*{0.2cm}
\caption{The STC algorithm for constructing a frame with prescribed spectrum and norms.}
\label{fig:STC}
\end{table}

We conclude our review of spectral tetris with a short example of the STC algorithm in execution.

\begin{example}
We run STC on the spectral tetris ready sequences of norms $\{a_n\}_{n=1}^4=\{1,\sqrt{3},\sqrt{2},\sqrt{2}\}$ and eigenvalues $\{\lambda_m\}_{m=1}^2=\{2,6\}$ to create a $2\times 4$ synthesis matrix $F$.  The cursor begins at $(1,1)$ where $a_1^2=1\leq 2=\lambda_1$, and so STC inserts a singleton.  The cursor moves to $(1,2)$.
\begin{center} \begin{tabular}{cccc||c} $1$ & $3$ & $2$ &$2$	& \\ \hline\hline 
																			$\odot$&* &*&*& $2$ \\ 
																			*&*&*&*&$6$ 
								\end{tabular}
	$\rightarrow$
								\begin{tabular}{cccc||c} $1$ & $3$ & $2$ &$2$	& \\ \hline\hline 
																			$1$&$\odot$ &*&*& $2$ \\ 
																			$0$&*&*&*&$6$ 
								\end{tabular}
		
\end{center} 
Now (\ref{endrow}) is satisfied since $a_1^2=1\leq \lambda_1=2<a_1^2+a_2^2=4$.  Thus STC inserts two doubletons comprising a block of the form (\ref{block}).  The cursor moves to $(2,4)$.
\begin{center} 	\begin{tabular}{cccc||c} $1$ & $3$ & $2$ &$2$	& \\ \hline\hline 
																			$1$&$\odot$ &*&*& $2$ \\ 
																			$0$&*&*&*&$6$ 
								\end{tabular}
		$\rightarrow$
								\begin{tabular}{cccc||c} $1$ & $3$ & $2$ &$2$	& \\ \hline\hline 
																			$1$& $\sqrt{\frac{1}{3}}$ & $\sqrt{\frac{2}{3}}$ &$0$& $2$ \\ 
																			$0$& $\sqrt{\frac{8}{3}}$ & $-\sqrt{\frac{4}{3}}$ &$\odot$&$6$ 
								\end{tabular}
\end{center} 
Then $\sum_{n=1}^4 a_n^2=\sum_{m=1}^2\lambda_m=8$, and STC completes the synthesis matrix $F$ by adding a singleton.
\[
F= \begin{bmatrix}
																				1& \sqrt{\frac{1}{3}}& \sqrt{\frac{2}{3}}&0\\
																				0&\sqrt{\frac{8}{3}}& -\sqrt{\frac{4}{3}}&\sqrt{2}
																			 \end{bmatrix}
\]
\end{example}

\subsection{Spectral Tetris Fusion Frames}
As a frame operator is a sum of rank one projections, fusion frames are a natural generalization of a frame of vectors to a frame of subspaces.  For an $M$-dimensional Hilbert space $\cH_M$, subspaces $\{W_k\}_{k=1}^K$, and positive weights $\{v_k\}_{k=1}^K$, $\{W_k,v_k\}_{k=1}^K$ is a {\em fusion frame} for $\cH_M$ if there are constants $0<A\leq B < \infty$ so that
\[ 
	A \|f\|^2 \le \sum_{k=1}^K v_k^2\|P_kf\|^2 \le B \|f\|, \mbox{ for all } f\in \cH_M,
\]
where $P_k$ is the orthogonal projection onto $W_k$.  We call $A,B$ the {\em fusion frame bounds}, and if $A=B$, this is a {\em tight fusion frame}.  The fusion frame operator $S:\cH_M \rightarrow \cH_M$ is then given by $S=\sum_{k=1}^K \nu_k^2 P_k$.  

Since spectral tetris outputs conventional frames, we need a connection between these and fusion frames.

\begin{theorem}\label{tight}
For $k\in\{1,\ldots,K\}$, let $\nu_k>0$, let $W_k$ be a subspace of $\RR^M$, and let $\{f_{k,j}\}_{j=1}^{d_k}$ be a tight frame for $W_k$ with tight frame bound $\nu_k^2$.  Then the following are equivalent.
\begin{enumerate}[(a)]
	\item $\{W_k,\nu_k\}_{k=1}^K$ is a fusion frame whose fusion frame operator has spectrum $\{\lambda_m\}_{m=1}^M$.
	\item $\{f_{k,j}\}_{k=1,j=1}^{K,d_k}$ is a frame whose frame operator has spectrum $\{\lambda_m\}_{m=1}^M$.
\end{enumerate}
\end{theorem}

\begin{proof}
Let $S$ be the fusion frame operator for $\{W_k,\nu_k\}_{k=1}^K$ and $S'$ the frame operator for $\{f_{k,j}\}_{k=1,j=1}^{\ K,\ \ d_k}$.  Letting $P_k$ be the orthogonal projection onto $W_k$, we have for any vector $f\in \RR^M$,
\begin{equation}\label{basis1}
	Sf=\sum_{k=1}^K \nu_k^2 P_k f = \sum_{k=1}^K \sum_{j=1}^{d_k} \inner{f}{f_{k,j}}f_{k,j}=S'f.
\end{equation}
\end{proof}
Often we choose orthonormal bases for each subspace $W_k$.  Indeed if $\{f_{k,j}\}_{j=1}^{d_k}$ is an orthonormal basis for $W_k$, then (\ref{basis1}) becomes
\begin{equation}\label{basis}
	Sf=\sum_{k=1}^K \nu_k^2 P_k f = \sum_{k=1}^K \sum_{j=1}^{d_k} \nu_k^2 \inner{f}{f_{k,j}}f_{k,j}= \sum_{k=1}^K \sum_{j=1}^{d_k} \inner{f}{{\nu_k}f_{k,j}}{\nu_k}f_{k,j},
\end{equation}
and thus every fusion frame arises from a conventional frame partitioned into equal-norm, orthogonal sets. 

In \cite{CFHWZ}, where the authors assume unit-norm vectors and unit weights, this leads to the definition of a {\em spectral tetris fusion frame}: a fusion frame $\{W_k,\nu_k\}_{k=1}^K$ arising from a spectral tetris frame $F=\{f_n\}_{n=1}^N$ and a partition $\{J_k\}_{k=1}^K$ of $\{1,\ldots,N\}$ so that $\{f_n\}_{n\in J_k}$ is an orthonormal basis for $W_k$.  This is a special case of Theorem \ref{tight}.  However, since our purpose is to drop the assumptions of unit-norm vectors and unit weights, it is not clear if similar orthogonality properties are sufficient to encompass all fusion frames arising from a spectral tetris construction.  Since STC allows us to set more arbitrary vector norms, in the more general case of fusion frames with non-unit weights, we instead give the following definition:

\begin{definition} Suppose $\{W_k,\nu_k\}_{k=1}^K$ is a fusion frame with frame operator $S$.  We say $\{W_k,\nu_k\}_{k=1}^K$ is a spectral tetris fusion frame if there exists a spectral tetris frame $F=\{f_n\}_{n=1}^N$ with frame opertator $S$, and if there exists a partition $\{J_k\}_{k=1}^K$ of $\{1,\ldots,N\}$ such that $\{f_n\}_{n\in J_k}$ is a tight frame for $W_k$ with tight frame bound $\nu_k^2$.  Further, we say $F$ and $\{J_k\}_{k=1}^K$ generate $\{W_k,\nu_k\}_{k=1}^K$.
\end{definition}

Since every fusion frame arises from a partition of a traditional frame, we introduce additional notation to easily identify subfamilies of frame vectors.  Given a frame $F=\{f_n\}_{n=1}^N$ and a subset $J\subseteq \{1,\ldots,N\}$, we denote the subfamily $\{f_n:n\in J\}=F_J$.  Since $F_J$ is a frame for its span, we again do not distiguish this set from its induced synthesis matrix.


\section{Weighted Fusion Frame Construction} \label{construction}

Given a sequence of eigenvalues for a fusion frame operator and a sequence of weights with corresponding dimensions, we wish to construct a fusion fusion frame with these properties.  Using spectral tetris, the spectrum for our construction is fixed, however, vector norms are not, and choices for our norms are not unique.  This is clear by the following simple example.

\begin{example}
Consider $\RR^2$ and a sequence of weights $(\sqrt{2},1)$ with corresponding subspace dimensions $(2,1)$.  Also, suppose we want the fusion frame operator to have eigenvalues $(2,3)$.  We use STC to produce a variety of frames whose frame operator has this spectrum:
\begin{enumerate}[(a)]
\item \label{a1} The sequence of norms $(\sqrt{2},\sqrt{2},1)$ produces the frame

\[
	\begin{bmatrix}
		f_1	&f_2	&f_3\\
	\end{bmatrix}
	=
	\begin{bmatrix}
		\sqrt{2}	&0	&0\\
		0	&\sqrt{2}	&1\\
	\end{bmatrix}.
\]
\item \label{b1} The sequence of norms $(1,1,\sqrt{2},1)$ produces the frame
\[
	\begin{bmatrix}
		g_1	&g_2	&g_3	&g_4\\
	\end{bmatrix}
	=
	\begin{bmatrix}
		1	&1	&0	&0\\
		0	&0	&\sqrt{2}	&1\\
	\end{bmatrix}.
\]
\item \label{c1}The sequence of norms $\left(1,\sqrt{\frac{3}{2}},\sqrt{\frac{3}{2}},1\right)$ produces the frame
\[
	\begin{bmatrix}
		h_1	&h_2	&h_3	&h_4\\
	\end{bmatrix}
	=
	\begin{bmatrix}
		1	&\sqrt{\frac{1}{2}}	&\sqrt{\frac{1}{2}}	&0\\
		0	&1	&-1	&1\\
	\end{bmatrix}.
\]
\end{enumerate}
A fusion frame $\{W_k,\nu_k\}_{k=1}^2$, $\nu_1=\sqrt{2},\nu_2=1$, is then obtained via STC by defining $W_1 = \spann(f_1,f_2)$, $W_2=\spann(f_3)$ or $W_1 = \spann(g_1,g_2,g_3)$, $W_2=\spann(g_4)$ or $W_1 = \spann(h_1,h_2,h_3)$, $W_2=\spann(h_4)$.  All three generate the same fusion frame.
\end{example}

The differences amoung constructions in this example are superficial; (\ref{b1}) simply splits a vector from (\ref{a1}) into two colinear vectors, and (\ref{c1}) takes two orthogonal vectors from (\ref{b1}) and combines them into a $2\times 2$ block spanning the same $2$-dimensional space.  In fact, all spectral tetris frames which generate a given fusion frame are related in this manner.  Before we state and prove this more formally as Theorem \ref{thm1}, we first give a useful proposition.

\begin{proposition} \label{prop1}
Let $F=\{f_n\}_{n=1}^N$ be an $M\times N$ spectral tetris frame.  Suppose $J\subseteq \{1,\ldots,N\}$ such that $F_J$ is a tight frame for $W_J=\spann(F_J)$.  Let $f_j\in F_J$ be a doubleton; without loss of generality, say $f_j,f_{j+1}$ contain a $2\times 2$ block.  If there exists some $f_{j'} \in F_J$, $j\neq j'$ such that $\inner{f_j}{f_{j'}}\neq 0$, then $f_{j+1}\in F_J$.
\end{proposition}

\begin{proof}
Suppose $\inner{f_j}{f_{j'}}\neq 0$.  Let $r,r+1$ denote the two rows over which $f_j, f_{j+1}$ contain a block.  We will assume $j'<j$ as the other case is proven similarly.  Let $\{e_m\}_{m=1}^M$ be the eigenvectors of the frame operator $FF^*$ indexed in the same order as their respective eigenvalues in STC.  We consider two cases.

Case I: Suppose $f_{j'}$ is a singleton.  Since $\inner{f_j}{f_{j'}}\neq 0$ and $j'<j$, this singleton appears in row $r$, and then $\spann(e_r,e_{r+1})\subseteq W_J$.  Let $P$ be the orthogonal projection onto $\spann(e_r,e_{r+1})$.  Since $F_J$ is a tight frame for $W_J$, $PF_J$ is a tight frame for $\spann(e_r,e_{r+1})$.  Specifically $PF_J(PF_J)^*=cI$ on $\spann(e_r,e_{r+1})$ implying rows $r$ and $r+1$ of $F_J$ orthogonal.  Due to sparcity of STC, orthogonality can only be achieved by including the other half of the block associated with $f_j$.  That is $f_{j+1}\in F_J$.

Case II: Suppose $f_{j'}$ is a doubleton, and without loss of generality let $f_{j'+1}$ denote the other doubleton of the $2\times 2$ block.  We may further assume there does not exist any $f_n \in F_J$ which is a singleton in rows $r$ or $r+1$, for otherwise Case I would apply.  Our goal will be to show $\spann(e_r,e_{r+1})\subseteq W_J$ still holds.  Then we can project $F_J$ onto $\spann(e_r,e_{r+1})$, consider the resulting tight frame, and the result will follow as in Case I.

Note $f_{j'},f_{j'+1}$ form a block over rows $r-1$ and $r$.  If $f_{j'+1}\in F_J$ then $\spann(e_{r-1},e_{r})\subseteq W_J$, and all together $f_{j'},f_{j'+1},f_j\in F_J$ implies 
\begin{equation}\label{build1}
	\spann(e_{r},e_{r+1})\subseteq \spann(e_{r-1},e_{r},e_{r+1})\subseteq W_J.
\end{equation}
This was our goal; we would be done.  

So suppose $f_{j'+1}\notin F_J$; Case I now necessitates there are also no singletons in row $r-1$ of $F_J$.  Let $P_1$ be the orthogonal projection onto $\spann(f_j,f_{j'})\subseteq W_J$, and consider the tight frame $P_1F_J$.  Now $P_1F_J$ contains a non-zero vector $P_1f_q$, $q\neq j,j'$ for otherwise $P_1F_J$ is tight frame consisting of $2$ non-zero vectors, $f_j$ and $f_{j'}$, spanning a $2$ dimensional space.  This would require $f_j,f_{j'}$ be orthogonal, a contradiction.  Further $f_q$ must be a doubleton since $P_1$ projects onto a subspace of $\spann(e_{r-1},e_r,e_{r+1})$, and $F_J$ has no singletons in rows $r-1,r,r+1$.  So $f_q\in F_J$ and supposing $f_{q+1}$ completes a block with $f_q$, this block must occur over rows $r+1,r+2$ or $r-2,r-1$.  We assume the latter as the other case is handled similarly.  Now if $f_{q+1}\in F_J$, then $\spann(e_{r-2},e_{r-1})\subseteq W_J$, and $f_q,f_{q+1},f_{j'}\in F_J$ give
\begin{equation}\label{build2}
	\spann(e_{r-1},e_{r})\subseteq \spann(e_{r-2},e_{r-1},e_r)\subseteq W_J
\end{equation}
Then (\ref{build2}) and $f_{j'},f_{j'+1},f_j\in F_J$ combine exactly as in (\ref{build1}).  We would be done.

So suppose $f_{j'+1},f_{q+1} \notin F_J$; Case I again implies there are also no singletons in row $r-2$.  We define $P_2$ as the orthogonal projection onto $\spann(f_j,f_{j'},f_q)\subseteq W_J$ and consider the tight frame $P_2F_J$.  $P_2F_J$ contains a non-zero $P_2f_{q'}$, $q'\neq q,j,j'$.  Otherwise $P_2F_J$ is a tight frame with $3$ nonzero vectors, $f_j,f_{j'},f_g$, spanning a $3$-dimensional space requiring orthogonality, a contradiction.  As before, $f_{q'}$ must be a doubleton since rows $r-2$ through $r$ have no singletons.

Continuing this line of reasoning, at each step we have either $\spann(e_{r},e_{r+1}) \subseteq W_J$ or a projection $P_z$ onto a subspace of $F_J$ where $P_zF_J$ is a tight frame with at least $z+1$ non-zero vectors.  Since our family of vectors is finite, we must have $\spann(e_{r},e_{r+1})\subseteq W_J$.
\end{proof}

\begin{theorem}\label{thm1}
If $\{W_k, {a_k}\}_{k=1}^K$ is a spectral tetris fusion frame in $\RR^M$, there exists a spectral tetris frame $F=\{f_n\}_{n=1}^N$ and a partition $\{J_k\}_{k=1}^K$ of $\{1,\ldots,N\}$ generating this fusion frame such that $\norm{f_n}={a_{k}}$ and $\inner{f_n}{f_{n'}}=0$, $n,n'\in J_k$ for each $k \in \{1,\ldots, K\}$.
\end{theorem}

\begin{proof}
Since $\{W_k, {a_k}\}_{k=1}^K$ is a spectral tetris fusion frame, there exists an ordering of the eigenvalues for the frame operator $\{\lambda_m\}_{m=1}^M$ and a sequence of norms $\{{b_{n}}\}_{n=1}^{N'}$ such that STC produces a frame $G=\{g_{n}\}_{n=1}^{N'}$, which along with a partition $\{J_{k}\}_{k=1}^K$ of $\{1,\ldots,N'\}$ generates the given fusion frame.  Leaving the eigenvalues unchanged, we will modify $\{{b_{n}}\}_{n=1}^{N'}$ so that STC constructs a frame with the properties we desire while generating the same fusion frame.  Since the frame operator is given by
\[
	GG^*= \sum_{k=1}^K G_{J_k}G^*_{J_k},
\]
we we may work within each $G_{J_k}$ individually.  We consider three cases.

Case I: Suppose $G_{J_k}$ were an orthogonal set.  Since each $G_{J_k}$ is also a tight frame for $W_k$ with tight frame bound $a_k^2$, this implies $G_{J_k}$ is also equal-norm with $\norm{f_n}={b_k}={a_k}$, $n\in J_k$.  In this case, there is nothing to change.

Case II: Suppose $G_{J_k}$ contains vectors $g_i$ and $g_{j}$ which are colinear with $i<j$.  We replace the norms ${b_i},{b_j}$ with the single norm $({b_i}^2+{b_j}^2)^{1/2}$ and run STC.  This produces same frame except with $g_i,g_j$ replaced by a single vector, say $f$, colinear with $g_i,g_j$ and with norm  $({b_i}^2+{b_j}^2)^{1/2}$.  We now have
\[
	G'_{J_k}=G_{J_k}\setminus \{g_i , g_j\} \cup \{f\}.
\]
Comparing $G_{J_k}$ and $G'_{J_k}$, it is not difficult to see row inner products and row norms are unchanged.  Thus
\begin{equation}\label{unchanged}
	G'_{J_k}(G'_{J_k})^*=G_{J_k}G_{J_k}^*.
\end{equation}

Case III: Suppose $G_{J_k}$ contains a pair of vectors $g_i,g_j$, $i<j$ which are neither orthogonal nor colinear.  Then $g_i$ or $g_j$ is a doubleton (or possibly both).  Without loss of generality, let $g_i$ be a doubleton and $g_i,g_{i+1}$ contain a block $A(x,b_i,b_{i+1})$.  By Proposition \ref{prop1}, $g_{i+1}\in G_{J_k}$.  Letting $y=a_i^2+a_{i+1}^2-x$, replace the norm ${b_i}$ with $x$ and ${b_{i+1}}$ with $y$.  Then spectral tetris produces the same frame except with $g_i,g_{i+1}$ replaced by new vectors, say $f_i$ and $f_{i+1}$, disjointly supported singletons with norms $x$ and $y$ respectively. We now have
\[
	G'_{J_k}=G_{J_k}\setminus \{g_i , g_{i+1}\} \cup \{f_i,f_{i+1}\}.
\]  
Since $A(x,b_i,b_{i+1})$ had orthogonal rows by construction, we observe the row inner products and row norms of $G'_{J_k}$ compared to $G_{J_k}$ are unchanged.  We again have (\ref{unchanged}).

By applying Case III we gain orthogonality of frame vectors, possibly at the cost of added colinearity; we remove colinearity by Case II.  By iteratively applying Cases II and III, we arrive at a fusion frame $\{W_k,{a_k^2}\}_{k=1}^K$ generated by a spectral tetris frame $F=\{f_n\}_{n=1}^N$ and a partition $\{J_k\}_{k=1}^K$ such that each $F_{J_k}$ satisfies Case I.  That is, each $F_{J_k}$ is an orthogonal equal-norm set.
\end{proof}

Since we have shown every spectral tetris fusion frame can be generated by partitioning a spectral tetris frame into equal norm, orthogonal vectors, we are now able to give necessary and sufficient conditions for constructing fusion frames via spectral tetris.

\begin{theorem} \label{necessary and sufficient}
Let $\{{a_k}\}_{k=1}^K$ be a sequence of weights, $\{\lambda_m\}_{m=1}^M$ a sequence of eigenvalues, and $\{d_k\}_{k=1}^K$ a sequence of dimensions.  Let $N=\sum_{k=1}^{K}d_k$, and now consider each ${a_k}$ repeated $d_k$ times.  We will use a double index to reference specific weights and a single index to emphasize the ordering:
\[
	\{a_{k,j}\}_{k=1,j=1}^{\ K,\ \ d_k}=\{{a}_n\}_{n=1}^N.
\]
Then spectral tetris can construct a fusion frame whose subspaces have the given weights and dimensions, and whose frame operator has the given spectrum if and only if there exists a spectral-tetris-ready permutation of $\{{a}_n\}_{n=1}^N$ and $\{\lambda_m\}_{m=1}^M$, say $\{{a}_{\sigma n}\}_{n=1}^N$ and $\{\lambda_{\sigma'm}\}_{m=1}^M$  whose associated partition $1\leq n_1,\leq \cdots, \leq n_M = N$ satisifes
\begin{enumerate}[(A)]
	\item if $\sum_{n=1}^{n_k}{a}_{\sigma n}^2 < \sum_{m=1}^k\lambda_{\sigma' m}$, then \label{A}			
			\begin{enumerate}
				\item if $\sum_{n=1}^{n_{k+1}}{a}_{\sigma n}^2 < \sum_{m=1}^{k+1}\lambda_{\sigma' m}$, then for $a_{i,j},a_{p,q}\in \{{a}_{\sigma n}\}_{n=n_k}^{n_{k+1}+1}$, $j\neq q$ \label{Aa}
				\item if $\sum_{n=1}^{n_{k+1}}{a}_{\sigma n}^2 = \sum_{m=1}^{k+1}\lambda_{\sigma' m}$, then for $a_{i,j},a_{p,q}\in \{{a}_{\sigma n}\}_{n=n_k}^{n_{k+1}}$, $j\neq q$ \label{Ab}
			\end{enumerate}
	\item if $\sum_{n=1}^{n_{k}}{a}_{\sigma n}^2 = \sum_{m=1}^k\lambda_{\sigma' m}$, then \label{B}
			\begin{enumerate}
				\item if $\sum_{n=1}^{n_{k+1}}{a}_{\sigma n}^2 < \sum_{m=1}^{k+1}\lambda_{\sigma' m}$, then for $a_{i,j},a_{p,q}\in \{{a}_{\sigma n}\}_{n=n_k+1}^{n_{k+1}+1}$, $j\neq q$ \label{Ba}
				\item if $\sum_{n=1}^{n_{k+1}}{a}_{\sigma n}^2 = \sum_{m=1}^{k+1}\lambda_{\sigma' m}$, then for $a_{i,j},a_{p,q}\in \{{a}_{\sigma n}\}_{n=n_k+1}^{n_{k+1}}$, $j\neq q$ \label{Bb}
			\end{enumerate}
\end{enumerate}
for all $k=1,\ldots,M-1$.
\end{theorem}

\begin{proof}
Notice $(\ref{A})$ applies when STC inserts a block at norm $a_{\sigma n_k}$.  Then $(\ref{Aa})$ applies if a block is also needed at $a_{\sigma n_{k+1}}$.  In this case, $\{{f}_{\sigma n}\}_{n=n_k}^{n_{k+1}+1}$ is the maximal set of pairwise non-orthogonal vectors containing $f_{{\sigma n_k}}$.  Thus the statement $(\ref{Aa})$ requires the corresponding set of norms $\{{a}_{\sigma n}\}_{n=n_k}^{n_{k+1}+1}$ contains no two norms $a_{i,j},a_{p,q}$ where $j=q$.  
Then $(\ref{Ab}),(\ref{Ba})$, and ($\ref{Bb})$ are simlilar statements which cover the remaining cases for when STC could insert blocks or singletons at $a_{\sigma n_k}$ and $a_{\sigma (n_{k+1})}$.  All together these statments simply state if STC constructs a frame from norms $\{{a}_{\sigma n}\}_{n=1}^N$ and eigenvalues $\{\lambda_{\sigma'm}\}_{m=1}^M$, and $f,g$ are two frame elements associated with norms $a_{i,j},a_{p,q}$ with $j=q$, then $\inner{f}{g}=0$.

With this in mind, the result follows since a spectral-tetris-ready ordering is necessary and sufficient to build a spectral tetris frame with the desired spectrum by Theorem \ref{ready}.  Then a spectral tetris frame with orthogonality amoung equal-norm vectors is clearly sufficient to generate the desired fusion frame by Theorem \ref{tight} (see (\ref{basis}) in particular).  Finally, the existence of such a spectral tetris frame is also necessary by Theorem \ref{thm1}.
\end{proof}


\section{Concrete Construction}\label{concrete}
While Theorem \ref{necessary and sufficient} provides precise conditions for when a fusion frame can be constructed via spectral tetris, it is difficult to apply for actual constructions.  Indeed, given a spectrum for a fusion frame operator and a sequence of weights repeated appropriately for subspace dimensions, how does one find an appropriate spectral-tetris-ready ordering?  In this section we consider several special cases in which an appropriate ordering may be found easily leading to concrete constructions.

We begin with a proposition which demonstrates how we may obtain a spectral-tetris-ready ordering when the given norms are small compared to the prescribed eigenvalues.

\begin{proposition} \label{gen}
Given a sequence of norms $\{{a_n}\}_{n=1}^N$, and a sequence of eigenvalues $\{\lambda_m\}_{m=1}^M$ where $\sum_{n=1}^N a_n^2 = \sum_{m=1}^M \lambda_m$, if 
\begin{equation} \label{allow}
	\max_{i,j\in\{1,\ldots,N\}}(a_i^2+a_j^2)\leq \min_{m\in \{1,\ldots,M\}}\lambda_m,
\end{equation}
then the sequences can be made spectral-tetris-ready by systematically switching adjacent weights.
\end{proposition}

\begin{proof}
Using the sequences $\{a_n\}_{n=1}^N$ and $\{\lambda_m\}_{m=1}^M$, we need to show if STC has its cursor at $(m,n)$ and cannot continue, interchanging $a_{n}$ and $a_{n+1}$ allows the algorithm to proceed.

Suppose the cursor is in row $r$.  STC will insert singletons until column $\ell$ when (\ref{endrow}) is satisfied.  Recall if the first inequality holds as an equality in (\ref{endrow}), STC inserts a singleton and the cursor proceeds to row $r+1$; the algorithm continues to run.  If (\ref{endrow}) does not hold as on equality, STC will attempt to create a $2\times 2$ block, $A=A(x,a_{\ell+1},a_{\ell+2})$ of the form (\ref{block}).  We must ensure such an $A$ exists and this $A$ allows STC to continue.

Assuming for the moment such a block exists, STC inserts $A$ and the cursor moves to row $r+1$.  STC continues as long $y=a_\ell^2+{a_{\ell+1}^2}-x \leq \lambda_{r+1}$, but this is always satisfied since
\[
	y\leq a_{\ell}^2+a_{\ell+1}^2\leq  \max_{i,j\in\{1,\ldots,N\}}(a_i^2+a_j^2)\leq \min_{m\in \{1,\ldots,M\}}\lambda_m \leq \lambda_{r+1}.
\]
Thus we need only concern ourselves with the existence of $A(x,a_\ell,a_{\ell+1})$.  That is, we must satisfy (\ref{bexist1}) and (\ref{bexist2}).

Note $x<a_{\ell}^2$ for otherwise
\[
	\sum_{n=1}^{\ell}a_n^2\leq \sum_{n=1}^{\ell-1}a_n^2 +x =  \sum_{m=1}^r \lambda_m
\]
violating (\ref{endrow}).  Thus (\ref{bexist1}) is always satisfied, and (\ref{bexist2}) holds if we also have $x< a_{\ell+1}^2$.  This implies STC can insert a block if
\begin{equation}\label{blockex}
	\sum_{m=1}^r \lambda_m=\sum_{n=1}^{\ell-1}a_n^2+x< \sum_{n=1}^{\ell-1} a_n^2 + a_{\ell+1}^2.
\end{equation}
Now suppose such a block does not exist so that $\sum_{m=1}^r \lambda_m \geq \sum_{n=1}^{\ell-1} a_n^2 + a_{\ell+1}^2$.  Taking the original sequence $\{a_n\}_{n=1}^N$ and interchanging norms $a_{\ell}, a_{\ell+1}$, we re-index the new ordering as $\{a_m'\}_{m=1}^M$ and have
\[
	\sum_{n=1}^{\ell} (a'_n)^2 = \sum_{n=1}^{\ell-1} a_n^2 +a_{\ell+1}^2 \leq \sum_{m=1}^r \lambda_m \mbox{ and } \sum_{m=1}^r \lambda_m < \sum_{n=1}^{\ell+1} a_n^2 = \sum_{n=1}^{\ell+1} (a'_n)^2.
\]
STC now inserts the singleton $a_{\ell}'$ at $(r,\ell)$ since the block is now required at column $\ell+1$ according to (\ref{endrow}).  The cursor continues to $(r, \ell+1)$ where STC now attempts to insert a block $A'=A'(x',a'_{\ell+1}, a'_{\ell+2})$.  Similar to (\ref{blockex}), $A'$ exists if
\[
	\sum_{m=1}^r \lambda_m < \sum_{n=1}^{\ell} (a'_n)^2 + (a'_{\ell+2})^2= \sum_{n=1}^{\ell-1} a_n^2 + a_{\ell+1}^2+ a_{\ell+2}^2.
\]
If $A'$ also fails to exist, switch $a_{\ell+1}'$ with $a_{\ell+2}'$, re-index, and STC will insert another singleton.  Continuing this line of reasoning, an appropriate block will exist when the cursor reaches column $(\ell+p)$ if
\begin{equation}\label{blockexist}
	\sum_{m=1}^r \lambda_m < \sum_{n=1}^{\ell-1} a_n^2 + \sum_{n=\ell+1}^{\ell+p}a_n^2= \sum_{n=1}^{\ell+p}a_n^2 - a_{\ell}^2.
\end{equation}
Finally, such a $p$ must exist due to the trace condition $\sum_{m=1}^M \lambda_m = \sum_{n=1}^N a_n^2$.
\end{proof}

Proposition \ref{gen} is a modification of STC which allows the algorithm to handle non-spectral-tetris ready orderings.  Moreover, the process is incredibly simple to implement: insert Table \ref{fig:MSTC} between lines 13 and 14 in the STC algorithm (Table \ref{fig:STC}).  For convenience, we will refer to this procedure as spectral tetris re-ordering (STR).

\begin{table}[h]
\centering
\framebox{
\begin{minipage}[h]{6.0in}
\vspace*{0.3cm}
{\sc \underline{STR: Spectral Tetris Re-Ordering Procedure}}

Call procedure between lines 13 and 14 of STC.
\vspace*{0.4cm}

{\bf Parameters:}\\[-3ex]
\begin{itemize}
\item Dimension $M\in\NN$.
\item Number of frame elements $N\in\NN$.
\item Eigenvalues $\{\lambda_m\}_{m=1}^M$  and vector norms $\{a_n\}_{n=1}^N$ such that 
$\sum_{n=1}^N a_n^2 = \sum_{m=1}^M \lambda_m$ and $\max_{i,j\in\{1,\ldots,N\}}(a_i^2+a_j^2)\leq \min_{m\in \{1,\ldots,M\}}\lambda_m$
\end{itemize}

{\bf Algorithm:}\\[-3ex]
\begin{itemize}
\item[1)] If $\lambda_m>a_{n+1}^2,$ then
\item[2)] \hspace*{0.5cm} temp$=a_n$.
\item[3)] \hspace*{0.5cm} $a_{n+1}=a_n$.
\item[4)] \hspace*{0.5cm} $a_{n+1}=$temp.
\item[5)] \hspace*{0.5cm} Go to STC (5).
\item[6)] end.
\end{itemize}
\end{minipage}
}
\vspace*{0.2cm}
\caption{Procedure for running STC on a non-spectral-tetris-ready ordering}
\label{fig:MSTC}
\end{table}

With a procedure that always results in a spectral tetris ready ordering, we next find sufficient conditions for when STC/STR runs and maintains orthogonality conditions $(\ref{Aa},\ref{Ab},\ref{Ba},\ref{Bb})$ from Theorem \ref{necessary and sufficient}.

\begin{theorem} \label{construct}
Consider $\RR^M$ and a sequence of weights ${a_1}\leq {a_2}\leq \ldots \leq {a_k}$ with corresponding subspace dimensions $\{d_k\}_{k=1}^K$, and a sequence of eignvalues $\lambda_1\leq \ldots \leq \lambda_M$.  Let the doubley indexed sequence $\{a_{k,j}\}_{k=1,j=1}^{\ K,\ \ d_k}$ represent $a_k$ each repeated $d_k$ times.  Now STC/STR will build a weighted fusion frame $\{W_k,a_k\}_{k=1}^K$, $\dim(W_k)=d_k$ whose frame operator has the given spectrum if there exists an ordering $\{a_n\}_{n=1}^N$ of $\{a_{k,j}\}_{k=1,j=1}^{\ K,\ \ d_k}$ such that
\begin{enumerate}[(a)]
	\item $\sum_{n=1}^N a_n^2=\sum_{m=1}^M \lambda_m$ \label{aa}
	\item $a_{K-1,1}^2+a_{K,1}^2 \leq \lambda_1$ \label{bb}
	\item If $a_{\ell}=a_{k,j}$, $a_{\ell'}=a_{k',j'}$ with $j=j'$ and $\ell<\ell'$, then $\sum_{n=\ell}^{\ell'-1}a_n^2\geq 2 \lambda_M$ \label{cc}.
\end{enumerate}
\end{theorem}

\begin{proof}
Due to the trace condition (\ref{aa}), condition (\ref{bb}), and Proposition \ref{gen}, STC/STR will take the orderings $\{a_n\}_{n=1}^N$, $\{\lambda_m\}_{m=1}^M$ and generate a frame $\{f_n\}_{n=1}^N$ with the desired spectrum.  In order to create subspaces $W_k$, we must show that given this generated frame, the final ordering of weights satisfies $(\ref{Aa},\ref{Ab},\ref{Ba},\ref{Bb})$ from Theorem \ref{necessary and sufficient}.  Recall these are orthogonality requirements that state any two vectors corresponding to a repeated weight are orthogonal.

Let $a_{\ell}=a_{k,j}$, $a_{\ell'}=a_{k',j'}$ with $j=j'$ and $\ell<\ell'$ be such a repeated weight.  By (\ref{bb}) and (\ref{cc}), if $s,t\in \{\ell,\ell+1,\ldots,\ell'-1\}$, then
\begin{equation}\label{sum}
	\sum_{n=\ell}^{\ell'-1}a_n^2 - (a_s^2+a_t^2)\geq 2\lambda_M-(a_s^2+a_t^2) \geq 2\lambda_M - \lambda_1 \geq \lambda_M.
\end{equation}
We must show $f_{\ell}$ and $f_{\ell'}$ are orthogonal.  Due to the sparcity of STC/STR in the eigenbasis of the frame operator, we need only show these vectors are disjointly supported.  Suppose the cursor is at $(r,\ell)$.  Worst case is if (\ref{endrow}) holds so that STC requires a block, for then $f_\ell$ shares support with any $f_n$ with a non-zero entries in rows $r$ or $r+1$.  By STR, such a block will be inserted over weights $a_\ell$ and $a_{\ell+p}$ for the smallest $p$ satisfying (\ref{blockexist}).   Now we must show nomatter how STR switches weights, when the cursor reaches the column associated with $a_{\ell'}$, the cursor is below row $r+1$.

Case I:  Assume (\ref{blockexist}) is satisfied for $p=1$ so that $a_{\ell}$ does not shift.  Notice (\ref{endrow}) and (\ref{sum}) combine to produce
\[
	\sum_{n=1}^{\ell'-2}a_n^2 = \sum_{n=1}^{\ell}a_n^2 + \sum_{n=\ell}^{\ell'-1}a_n^2 -(a_{\ell'-1}^2+a_{\ell})^2 \geq \sum_{m=1}^r\lambda_m + \lambda_M \geq \sum_{m=1}^{r+1}\lambda_m.
\]
Also, for any $a_s\in \{a_n\}_{n=\ell+1}^{\ell'-1}$, 
\[
	\sum_{n=1}^{\ell'-1}a_n^2 - a_s^2 = \sum_{n=1}^{\ell}a_n^2 + \sum_{n=\ell}^{\ell'-1}a_n^2 - (a_s^2 + a_\ell^2) \geq \sum_{m=1}^{r+1}\lambda_m.
\]
Relating these inequalities to (\ref{endrow}) and (\ref{blockexist}), the cursor exits row $r+1$ at or before weight $a_{\ell'-1}$.

Case II:  Suppose (\ref{blockexist}) is satisfied for some $p\neq 1$.  By (\ref{endrow}), (\ref{sum}), and (\ref{aa})
\[
	\sum_{n=1}^{\ell'-3}a_n^2-a_\ell^2=\sum_{n=1}^{\ell}a_n^2+\left(\sum_{n=\ell}^{\ell'-1}a_n^2 -(a_{\ell'-1}^2+a_{\ell'-2}^2)\right)-2a_\ell^2>\sum_{m=1}^r\lambda_m+\lambda_M-\lambda_1\geq \sum_{m=1}^r\lambda_m,
\]
and thus $p\leq \ell'-3$.  This implies the cursor exits row $r$ at or before $a_{\ell'-3}$.  If $p=\ell'-3$ there remains room for a $2\times 2$ block to exist across $a_{\ell'-2},a_{\ell'-1}$; such a block or a singleton still finishes row $r+1$ at or before $a_{\ell'-1}$ since the same inequalities from Case I hold.

Thus after running STC/STR, the orthogonality conditions $(\ref{Aa},\ref{Ab},\ref{Ba},\ref{Bb})$ from Theorem \ref{necessary and sufficient} are met.  We have a spectral tetris frame $\{f_n\}_{n=1}^N=\{f_{kj}\}_{k=1,j=1}^{\ K,\ \ d_k}$ generating the desired fusion frame $\{W_k,a_k\}_{k=1}^K$ by setting $W_k=\spann(\{f_{k,j}\}_{j=1}^{d_k})$.

\end{proof}

One may now ask how to order the weights to achieve (\ref{cc}).  Intuitively we would want to space like-weights as far apart as possible in order to maximize $\sum_{n=\ell}^{\ell'-1}a_n$.  In the case of fusion frames with equi-dimensional subspaces, the best spacing is obvious, and the hypotheses of Theorem \ref{construct} simplify.  We start with a tight fusion frame with equi-dimensional subspaces.

\begin{corollary} 
Consider $\RR^M$ and a sequence of weights ${a_1}\leq {a_2}\leq \ldots \leq {a_k}$.  STC/STR can constuct a tight weighted fusion frame with the given weights, all subspaces of dimension $d$, provided
\begin{enumerate}[(i)]
	\item \label{req1} $a_{K-1}^2+a_K^2 \leq \lambda$
	\item \label{req2} $d/M \leq 1/2$
\end{enumerate}
\end{corollary}

\begin{proof}
Begin by repeating $a_k$ each $d$ times:
\[
	\{a_n\}_{n=1}^N=a_1,a_2,\ldots,a_{K},a_1,a_2,\ldots,a_{K},\cdots,a_1,a_2,\ldots,a_{K}
\]
where $N=dK$.  In this tight-case, the trace condition requires $M\lambda=\sum_{n=1}^N a_n^2$ which is condition (\ref{aa}) of Theorem \ref{construct}.  Condition (\ref{req1}) is precisely condition (\ref{bb}) of Theorem \ref{construct}.  Finally by (\ref{req2}) and the trace condition, for any $j=1,\ldots,N-K+1$ we have
\begin{equation*}
	\sum_{n=j}^{j+K-1}a_{n}^2=\sum_{n=1}^Ka_n^2\geq 2\frac{d}{M} \sum_{n=1}^K a_{n} = 2\lambda,
\end{equation*}
which is the last condition, (\ref{cc}) of Theorem \ref{construct}.  Apply the theorem.
\end{proof}

If we drop the tight-frame requirement, the following corrolary is an obvious similar application of Theorem \ref{construct}.

\begin{corollary} \label{cor2}
Consider $\RR^M$, a sequence of weights, ${a_1}\leq {a_2}\leq \ldots \leq {a_k}$, and a sequence of eigenvalues $\lambda_1\leq \ldots \leq \lambda_M$.  STC/STR can construct a weighted fusion frame with the given weights, all subspaces dimension $d$, and with the given spectrum provided
\begin{enumerate}[(i)]
	\item $d\sum_{n=1}^K a_n^2=\sum_{m=1}^M \lambda_m$ \label{aaa}
	\item $a_{K-1}^2+a_K^2 \leq \lambda_1$ \label{bbb}
	\item $\sum_{n=1}^K a_n^2\geq 2\lambda_M$ \label{ccc}
\end{enumerate}
\end{corollary}

\begin{remark}
In order for STC to build a desired fusion frame, a complex relationship amoung partial sums of weights, partial sums of eigenvalues, and dimensions of our subspaces must be satisfied according to Theorem \ref{necessary and sufficient}.  We simplified this relationship in Theorem \ref{construct} and its corollaries to achieve concrete constructions via STC/STR.  While these extra assumptions still allow a variety of fusion frames to be created, they are best suited for fusion frames with relatively flat spectrums.  For example, (\ref{aaa}) and (\ref{ccc}) of Corollary \ref{cor2} imply
\[
	\frac{\sum_{m=1}^M\lambda_m}{d}\geq 2\lambda_M,
\]
and this can clearly be manipulated to
\[
	\frac{\mbox{Average}(\{\lambda_m\}_{m=1}^M)}{2\lambda_M}\geq \frac{d}{M}.
\]
Hence if we desire STC/STR to garauntee the construction of fusion frames with relatively large subspaces, our prescribed frame operator must have a relatively flat spectrum.  However, the conditions
used here are of the correct order for practical applications.
That is, we generally do not work with large subspaces or with eigenvalues for a frame operator
which are very spread out.
\end{remark}

\end{document}